\documentclass[12pt,a4paper]{article}
\usepackage{latexsym,amssymb,amsfonts,amsmath,amsthm,nccmath,float,enumitem,setspace,bm,authblk}
\usepackage[usenames,dvipsnames]{xcolor}
\usepackage[hidelinks]{hyperref}
\usepackage[margin=2.45cm]{geometry}

\setlist{topsep=3pt,partopsep=0pt,itemsep=1pt,parsep=0pt}

\hypersetup{
    colorlinks,
    linkcolor={red!80!black},
    citecolor={blue!80!black},
    urlcolor={blue!80!black}
}

\newtheorem{Theorem}{Theorem}

\newtheorem{Lemma}{Lemma}

\def \leq {\leqslant}
\def \geq {\geqslant}

\allowdisplaybreaks

\let\oldproofname=\proofname
\renewcommand{\proofname}{\rm\bf{\oldproofname}}

\begin{document}

\title{Embedding arbitrary edge-colorings of hypergraphs into regular colorings}

\author[a]{Xiaomiao Wang}
\author[b]{Tao Feng}
\author[c]{Shixin Wang}

\affil[a]{School of Mathematics and Statistics, Ningbo University, Ningbo 315211, P.R. China}
\affil[b]{School of Mathematics and Statistics, Beijing Jiaotong University, Beijing 100044, P.R. China}
\affil[c]{UP FAMNIT, University of Primorska, Koper 6000, Slovenia}

\renewcommand*{\Affilfont}{\small\it}
\renewcommand\Authands{ and }
\date{}

\maketitle
\footnotetext[1]{E-mail address: wangxiaomiao@nbu.edu.cn} 
\footnotetext[2]{E-mail address: tfeng@bjtu.edu.cn; Supported by NSFC under Grant 12271023}
\footnotetext[3]{E-mail address: shixin.wang@famnit.upr.si; Supported by research program P1-0285 and research project J1-50000 by the Slovenian Research and Innovation Agency}

\begin{abstract}
For $\textbf{r}=(r_1,\ldots,r_k)$, an \textbf{r}-factorization of the complete $\lambda$-fold $h$-uniform $n$-vertex hypergraph $\lambda K_n^h$ is a partition of the edges of $\lambda K_n^h$ into $F_1,\ldots, F_k$ such that $F_j$ is $r_j$-regular and spanning for $1\leq j\leq k$. This paper shows that for $n>\frac{m-1}{1-2^{\frac{1}{1-h}}}+h-1$, a partial $\textbf{r}$-factorization of $\lambda K_m^h$ can be extended to an $\textbf{r}$-factorization of $\lambda K_n^h$ if and only if the obvious necessary conditions are satisfied.
\end{abstract}


\section{Introduction}\label{sec:intro}

A {\em hypergraph} $\mathcal{G}=(V,E)$ consists of a vertex set $V=V(\mathcal G)$, and an edge multiset $E=E(\mathcal G)$ which is a family of non-empty subsets of $V$. The \emph{multiplicity} of an edge $e$ in $\mathcal{G}$, written as $\textbf{mult}_{\mathcal{G}}(e)$, is the number of occurrences of $e$ in $\mathcal{G}$. The \emph{degree} $\textbf{deg}_{\mathcal{G}}(x)$ of a vertex $x\in V$ is the number of edges of $\mathcal{G}$ containing $x$. If every vertex degree in $\mathcal{G}$ is exactly $r$, then $\mathcal{G}$ is $r$-\emph{regular}. Let $\lambda K_n^h$ be the $n$-vertex hypergraph whose edge multiset $E$ is the collection of all $h$-subsets of the vertex set, each $h$-subset occurring exactly $\lambda$ times in $E$.

For $0<a<b$, denote by $[a,b]$ the set of integers from $a$ to $b$. A $q$-\emph{edge-coloring} of $\mathcal{G}=(V,E)$ is a mapping $f: E\rightarrow[1,q]$, and $f$ is {\em extended to} (or {\em embedded into}) a $k$-edge-coloring $g$ of $\mathcal{H}\supseteq\mathcal{G}$ if $f(e)=g(e)$ for each $e\in E$. For $\textbf{r}=(r_1,\ldots,r_k)$, an $\textbf{r}$-\emph{factorization} of $\mathcal{G}$ is a $k$-edge-coloring of $\mathcal{G}$ such that the color class $j$ for $j\in[1,k]$, written as $\mathcal{G}(j)$, induces an $r_j$-regular spanning subhypergraph of $\mathcal{G}$, called an {\em $r_j$-factor} of $\mathcal{G}$. $\mathcal{G}$ is said to be $\textbf{r}$-{\em factorable} if it admits an $\textbf{r}$-factorization. An $r$-factorization is an $\textbf{r}$-factorization with $r_1=\cdots=r_k=r$. For $\textbf{r}=(r_1,\ldots,r_k)$, a \emph{partial} $\textbf{r}$-\emph{factorization} of $\mathcal{G}$ is a $k$-edge-coloring of $\mathcal{G}$ where for every $j\in[1,k]$, the degree of each vertex of $\mathcal{G}(j)$ is at most $r_j$.

The problem of constructing $1$-factorization of $K_n^h$ dates back to the $18$-th century. In connection with Kirkman's famous Fifteen Schoolgirls Problem, Sylvester remarked in 1850 that $K_{15}^3$ is $1$-factorable. Clearly a necessary condition for $K_n^h$ admitting a 1-factorization is $h\mid n$. It turns out that this is also sufficient for $h=2$ (folklore) and $h=3$ (proved by Peltesohn \cite{p} in 1936). The sufficiency for general $h$ was eventually established by Baranyai \cite{Baranyai} in 1975. Furthermore, it was shown in \cite[Corollary 2]{Baranyai79} that $K_n^h$ is $r$-factorable if and only if $h\mid rn$ and $\frac{rn}{h}\mid \binom{n}{h}$.

Baranyai and Brouwer \cite{Baranyai-B} conjectured that a 1-factorization of $K_m^h$ can be extended to a 1-factorization of $K_n^h$ with $n>m$ if and only if $n\geq 2m$, $h\mid m$ and $h\mid n$. They proved this for $h=2,3$, and also for arbitrary $h$ when $n$ is sufficiently large. Their conjecture was settled by H\"{a}agkvist and Hellgren \cite[Theorem 2]{hh}. For general $r$, Bahmanian and Newman \cite[Theorem 1.7]{BN} showed that when $\gcd(m,n,h)=\gcd(m,h)$, an $r$-factorization of $K_m^h$ can be extended to an $r$-factorization of $K_n^h$  with $n>m$ if and only if $n\geq 2m$, $h\mid rm$, $h\mid rn$, $r\mid \binom{m-1}{h-1}$ and $r\mid \binom{n-1}{h-1}$.

Suppose that a partial $\textbf{r}$-factorization $P$ of $\lambda K_m^h$ is extended to an $\textbf{r}$-factorization of $\lambda K_n^h$, where $\textbf{r}=(r_1,\ldots,r_k)$. For $j\in[1,k]$, the existence of an $r_j$-factor in $\lambda K_n^h$ implies that $h\mid r_j n$. Since the degree of each vertex in $\lambda K_n^h$ is $\lambda\binom{n-1}{h-1}$, we have $\sum_{j=1}^kr_j=\lambda\binom{n-1}{h-1}$. Thus, in order to extend $P$ to an $\textbf{r}$-factorization of $\lambda K_n^h$, the following conditions are necessary.
\begin{align}\label{nec}
h\mid r_j n, \ \ \ \forall j\in[1,k],\ \ \ \ \emph{and} \ \ \ \sum_{j=1}^kr_j=\lambda\binom{n-1}{h-1}.
\end{align}
A quadruple $(n,h,\lambda,\textbf{r})$ is \emph{admissible} if it satisfies the above conditions. Bahmanian and Johnsen \cite[Theorem 1.1]{arxiv202209} proved that when $n\geq(h-1)(2m-1)$, these necessary conditions are also sufficient. This paper is devoted to improving the lower bound for $n$.

\begin{Theorem}\label{main-r-r}
For $n>\frac{m-1}{1-2^{\frac{1}{1-h}}}+h-1$ and $h\geq 2$, a partial $\textbf{r}$-factorization of $\lambda K_m^h$ can be extended to an $\textbf{r}$-factorization of $\lambda K_n^h$ if and only if $(n,h,\lambda,\textbf{r})$ is admissible.
\end{Theorem}

Comparing the two lower bounds for $n$, we have that $$(h-1)(2m-1)-(\frac{m-1}{1-2^{\frac{1}{1-h}}}+h-1)=(m-1)(2h-\frac{1}{1-2^{\frac{1}{1-h}}}-2),$$
which is strictly increasing for any given $m\geq 2$ and is greater than 0 for any $m\geq h\geq 3$. The problem of embedding a partial $r$-factorization of $K_m^h$ into an $r$-factorization of $K_n^h$ has been settled by Rodger and Wantland \cite{rw} for $h=2$ and $n\geq 2m$, by Bahmanian and Rodger \cite{br2013} for $h=3$ and $n\geq 3.414214m$, by Bahmanian \cite{a-jgt-2019} for $h=4$ and $n\geq4.847323m$, and for $h=5$ and $n\geq6.285214m$. Theorem \ref{main-r-r} unifies these results. One can check that our lower bound gives $n>2m-1$ for $h=2$, $n>3.41421m-1.41421$ for $h=3$, $n>4.84732m-1.84732$ for $h=4$, and $n>6.28521m-2.28521$ for $h=5$.

Section \ref{section-Preliminaries} provides a combinatorial identity and two combinatorial inequalities that are utilized in this paper. Theorem \ref{main-r-r} is proved in Section \ref{section-main}. Section \ref{sec:concluding} concludes the paper.

\section{Preliminaries} \label{section-Preliminaries}

The following combinatorial identity is often referred to as {\em Vandermonde's convolution}. For positive integers $n\geq m\geq h$,

\begin{align}\label{equ:comb_indetity}
\binom{n}{h}=\sum_{i=0}^{h}\binom{m}{h-i}\binom{n-m}{i}.
\end{align}

\begin{Lemma}\label{for h-2}
Let $n$, $m$ and $h$ be integers such that $n\geq m\geq h\geq 2$ and $n>\frac{m-1}{1-2^{\frac{1}{1-h}}}+h-1$. Then $2\binom{n-m}{h-1}>\binom{n-1}{h-1}.$
\end{Lemma}

\begin{proof} Since $n>\frac{m-1}{1-2^{\frac{1}{1-h}}}+h-1$,
\begin{align*}
\frac{\binom{n-m}{h-1}}{\binom{n-1}{h-1}}&=\frac{(n-m)!(n-h)!}{(n-1)!(n-m-h+1)!}=
\prod_{i=1}^{h-1}\frac{n-m-i+1}{n-i}=\prod_{i=1}^{h-1}\left(1-\frac{m-1}{n-i}\right)\\
&\geq \prod_{i=1}^{h-1}\left(1-\frac{m-1}{n-h+1}\right)=\left(1-\frac{m-1}{n-h+1}\right)^{h-1}>\frac{1}{2},
\end{align*}
and hence $2\binom{n-m}{h-1}>\binom{n-1}{h-1}$.
\end{proof}

\begin{Lemma}\label{inequality}
Let $n$, $m$ and $h$ be integers such that $n\geq m\geq h\geq 2$ and $n>\frac{m-1}{1-2^{\frac{1}{1-h}}}+h-1$. Then for every $i\in[1,h]$,  $$(h-i)\left(\sum_{l=1}^{i+1}\binom{m-1}{h-l}\binom{n-m}{l-1}\right)\leq\binom{n-1}{h-1},$$
where the equality holds if and only if $i=h-1$.
\end{Lemma}

\begin{proof} The case of $i=h$ is trivial. When $i=h-1$, it follows from \eqref{equ:comb_indetity} that
\begin{align}\label{eqn1}
\binom{n-1}{h-1}=\sum_{l=1}^{h}\binom{m-1}{h-l}\binom{n-m}{l-1}.
\end{align}
Assume that $i\in[1,h-2]$. Let $$f(i)=(h-i)\left(\sum_{l=1}^{i+1}\binom{m-1}{h-l}\binom{n-m}{l-1}\right)-\binom{n-1}{h-1}.$$
By Lemma \ref{for h-2},
\begin{align}\label{eqn:10-27}
f(h-2)&=2\left(\sum_{l=1}^{h-1}\binom{m-1}{h-l}\binom{n-m}{l-1}\right)-\binom{n-1}{h-1}\nonumber\\
&\overset{\eqref{eqn1}}{=}2\left(\binom{n-1}{h-1}-\binom{n-m}{h-1}\right)-\binom{n-1}{h-1}\nonumber\\
&=\binom{n-1}{h-1}-2\binom{n-m}{h-1}<0.
\end{align}
To prove $f(i)<0$ for any $i\in[1,h-3]$, which implies $h\geq 4$, we define $g(i):=f(i+1)-f(i)$ for every $i\in[1,h-3]$. If we can show that $g(i+1)>g(i)$ for $i\in[1,h-4]$ and $g(1)>0$, then $$g(h-3)>g(h-4)>\cdots>g(2)>g(1)>0,$$
and so $g(i)=f(i+1)-f(i)>0$ for every $i\in[1,h-3]$, which implies
$$f(h-2)>f(h-3)>\cdots>f(2)>f(1),$$
and hence it follows from $f(h-2)<0$ by \eqref{eqn:10-27} that $f(i)<0$ for any $i\in[1,h-3]$. Therefore, it suffices to show that $g(i+1)-g(i)>0$ for $i\in[1,h-4]$ and $g(1)>0$.

Since $h\geq i+4\geq 4$ and $n>\frac{m-1}{1-2^{\frac{1}{1-h}}}+h-1>hm$,
\begin{align*}
\frac{(h-i-2)\binom{m-1}{h-i-3}\binom{n-m}{i+2}}{(h-i)\binom{m-1}{h-i-2}\binom{n-m}{i+1}}
&=\frac{(h-i-2)^2(n-m-i-1)}{(h-i)(i+2)(m-h+i+2)}\nonumber\\
>\frac{hm-m-i-1}{(i+2)(m-h+i+2)}&>\frac{(h-1)(m-1)}{(i+2)(m-h+i+2)}>1,
\end{align*}
and hence
\begin{align*}
&g(i+1)-g(i)\\
=&(h-i-2)\binom{m-1}{h-i-3}\binom{n-m}{i+2}-(h-i)\binom{m-1}{h-i-2}\binom{n-m}{i+1}>0.
\end{align*}
Since $h\geq 4$ and $n-m>m-h+2$,
\begin{align*}
\frac{(h-2)\binom{m-1}{h-3}\binom{n-m}{2}}{\binom{m-1}{h-1}+(n-m)\binom{m-1}{h-2}}&=\frac{(n-m)(h-1)(h-2)^2(n-m-1)}{2(m-h+2)(m+(n-m-1)(h-1))}\nonumber\\
&\geq\frac{(n-m)(h-1)(h-2)^2}{2(m-h+2)h}>\frac{(h-1)(h-2)^2}{2h}>1,
\end{align*}
and so
\begin{align*}
g(1) & =f(2)-f(1)\\
& =(h-2)\binom{m-1}{h-3}\binom{n-m}{2}-\binom{m-1}{h-1}-(n-m)\binom{m-1}{h-2}>0.
\end{align*}
\end{proof}


\section{Proof of Theorem \ref{main-r-r}} \label{section-main}

We will use the following lemma to complete the proof of Theorem \ref{main-r-r}. The notation $x\approx y$ means $\lfloor y\rfloor\leq x\leq \lceil y\rceil$.


\begin{Lemma}\label{thm:B}{\rm \cite[Theorem 2.1]{a-jgt-2019}}
Let $\mathcal{H}$ be an ``unconventional'' $k$-edge-colored $t$-uniform hypergraph that allows edges to contain multiple copies of a vertex of $\mathcal{H}$. Then there exists a ``conventional'' $t$-uniform hypergraph $\mathcal F$ $($whose edges are all sets$)$ admitting a surjective mapping $\Psi: V (\mathcal F) \rightarrow V (\mathcal H)$ that naturally induces a bijection between $E(\mathcal F)$ and $E(\mathcal H)$, such that
\begin{enumerate}
\item[$(1)$] for each $u\in V(\mathcal H)$, each $v\in \Psi^{-1}(u)$ and $j\in [1,k]$,
$$\textbf{deg}_{\mathcal{F}(j)}(v) \approx \frac{\textbf{deg}_{\mathcal{H}(j)}(u)}{|\Psi^{-1}(u)|};$$
\item[$(2)$] for distinct $u_1,u_2,\ldots,u_s\in V(\mathcal H)$ and $U_l\subseteq \Psi^{-1}(u_l)$ with $|U_l|=m_l$ for $l\in [1,s]$, if $U_1\cup U_2\cup \cdots\cup U_s$ forms an edge of $\mathcal F$ $($which implies $\{u_1^{m_1},u_2^{m_2},\ldots,u_s^{m_s}\}$ forms an edge of $\mathcal H$$)$, then
$$\textbf{mult}_{\mathcal{F}}(U_1\cup U_2\cup \cdots\cup U_s)=\frac{ \textbf{mult}_{\mathcal{H}}(\{u_1^{m_1},u_2^{m_2},\ldots,u_s^{m_s}\}) }{\prod_{l=1}^s \binom{ |\Psi^{-1}(u_l)| }{m_l}}.$$
\end{enumerate}
where $u_l^{m_l}$ means $u_l$ occurs $m_l$ times for $l\in [1,s]$.
\end{Lemma}

In Lemma \ref{thm:B}, $\mathcal F$ is obtained by splitting each vertex of $\mathcal H$ into certain vertices in $\mathcal F$. The condition (1) guarantees that the degree of each vertex in each color class of $\mathcal H$ is shared evenly among the corresponding vertices in the same color class in $\mathcal F$. The condition (2) guarantees that the multiplicity of each edge in $\mathcal H$ is shared evenly among the corresponding edges in $\mathcal F$.

Now we adopt a similar yet more meticulous approach from \cite[Theorem 1.1]{arxiv202209} to prove Theorem \ref{main-r-r}.

\begin{proof}[\bf Proof of Theorem \ref{main-r-r}] The necessity is obvious. We examine its sufficiency. Let $\textbf{r}=(r_1,\ldots,$ $r_k)$ and let $\mathcal{G}$ be the given partially $\textbf{r}$-factorized $\lambda K_m^h$. We shall extend $\mathcal{G}$ to an $\textbf{r}$-factorized $\lambda K_n^h$, where $n>\frac{m-1}{1-2^{\frac{1}{1-h}}}+h-1$ and $(n,h,\lambda,\textbf{r})$ is admissible.

First we construct an ``unconventional'' hypergraph $\mathcal{H}$ that allows edges to contain multiple copies of a vertex of $\mathcal{H}$. Specifically, $\mathcal{H}$ is obtained by adding a new vertex $\alpha$ and $\sum_{i=1}^h\lambda\binom{m}{h-i}\binom{n-m}{i}$ new edges to $\mathcal{G}$ such that for each $i\in[1,h]$, $X\subseteq V(\mathcal{G})$ and  $|X|=h-i$,
$$\textbf{mult}_{\mathcal{H}}(X\cup\{\alpha^{i}\})=\lambda\binom{n-m}{i},$$
where $\alpha^{i}$ means $i$ copies of $\alpha$ and $X\cup\{\alpha^{i}\}$ is an edge in $\mathcal{H}$ consisting of $h$ vertices. For convenience, for some $X\subseteq V(\mathcal{G})$, we simply write $X\alpha^i$-edge or $*\alpha^i$-edge instead of the edge $X\cup\{\alpha^i\}$ in $\mathcal{H}$, and the $*\alpha^h$-edges in $\mathcal{H}$ are simply written as $\alpha^h$-edges. Note that the edges of $\mathcal{G}$ can be regarded as the $*\alpha^0$-edges. By assumption, all the $*\alpha^0$-edges have been colored by $k$ colors such that its color class $j$ for every $j\in[1,k]$ induces a spanning subhypergraph of $\mathcal{G}$ in which the degree of each vertex is at most $r_j$.

\textbf{Step $1$.} Assume that we have colored all the $*\alpha^0$-edges, $\ldots$, $*\alpha^{i-1}$-edges of $\mathcal{H}$ for some $i\in [1,h-1]$ in ascending order of $i$. We claim that we can greedily color all the $*\alpha^{i}$-edges of $\mathcal{H}$ so that $\textbf{deg}_{\mathcal{H}(j)}(x)\leq r_j$ for each $x\in V(\mathcal{G})$ and $j\in[1,k]$, where $\mathcal{H}(j)$ denotes the ``current'' color class $j$ of $\mathcal{H}$. 

Suppose by contradiction that some edge $X\cup\{\alpha^i\}$ in $\mathcal{H}$ cannot be colored, where $X$ is an $(h-i)$-subset of $V(\mathcal{G})$. Then for each $j\in[1,k]$, there is some $x\in X$ such that $\textbf{deg}_{\mathcal{H}(j)}(x)=r_j$ (note that $\mathcal H_{(j)}$ is the color class $j$ of $\mathcal H$ before the edge $X\cup\{\alpha^i\}$ is colored), and consequently, for all $j\in[1,k]$, $\sum_{x\in X}\textbf{deg}_{\mathcal{H}(j)}(x)\geq r_j$. On one hand,
\begin{align}\label{eqn:12-12-1}
\sum_{j=1}^{k}\sum_{x\in X}\textbf{deg}_{\mathcal{H}(j)}(x)\geq\sum_{j=1}^{k}r_j\overset{\eqref{nec}}{=}\lambda\binom{n-1}{h-1}.
\end{align}
On the other hand, the number of $*\alpha^{l}$-edges in $\mathcal{H}$ for each $l\in[0,i]$ containing any given vertex of $\mathcal G$ is $\lambda\binom{n-m}{l}\binom{m-1}{h-l-1}$, so counting the edges colored so far, we have
\begin{align*}
\sum_{j=1}^{k}\sum_{x\in X}\textbf{deg}_{\mathcal{H}(j)}(x)< & |X|\sum_{l=0}^i \lambda\binom{n-m}{l}\binom{m-1}{h-l-1} \\
= & \lambda(h-i)\left(\sum_{l=1}^{i+1}\binom{m-1}{h-l}\binom{n-m}{l-1}\right).
\end{align*}
Therefore,
\begin{align}\label{eqn:12-12}
\lambda\binom{n-1}{h-1}<\lambda(h-i)\left(\sum_{l=1}^{i+1}\binom{m-1}{h-l}\binom{n-m}{l-1}\right),
\end{align}
which contradicts Lemma \ref{inequality}.

Note that after coloring all the $*\alpha^{h-1}$-edges of $\mathcal{H}$, we have that for each $x\in V(\mathcal{G})$,
\begin{align}\label{eqn:12-6}
\sum_{j=1}^{k}\textbf{deg}_{\mathcal{H}(j)}(x)
=\lambda\left(\sum_{l=0}^{h-1}\binom{n-m}{l}\binom{m-1}{h-l-1}\right)
\overset{\eqref{eqn1}}{=}\lambda\binom{n-1}{h-1}
\overset{\eqref{nec}}{=}\sum_{j=1}^{k}r_j.
\end{align}
Since our coloring strategy requires $\textbf{deg}_{\mathcal{H}(j)}(x)\leq r_j$ for each $j\in[1,k]$, it follows from \eqref{eqn:12-6} that after coloring all the $*\alpha^{h-1}$-edges of $\mathcal{H}$, for each $x\in V(\mathcal{G})$ and $j\in[1,k]$,
\begin{align}\label{x}
\textbf{deg}_{\mathcal{H}(j)}(x)=r_j.
\end{align}

\textbf{Step $2$.} For $i\in [0,h-1]$ and $j\in[1,k]$, let $t_{ij}$ be the number of $*\alpha^i$-edges in  $\mathcal{H}$ colored $j$. Then by \eqref{x}, for $j\in[1,k]$,
\begin{align}\label{eqn:12-6-1}
r_jm=\sum_{i=0}^{h-1}(h-i)t_{ij}.
\end{align}
We now color all the $\alpha^{h}$-edges so that for each $j\in[1,k]$, the number of $\alpha^h$-edges in  $\mathcal{H}$ colored $j$ is
\begin{align}\label{eqn:12-7}
t_{hj}:=\frac{r_jn}{h}-\sum_{i=0}^{h-1}t_{ij}.
\end{align}
Note that since $(n,h,\lambda,\textbf{r})$ is admissible and $n>\frac{m-1}{1-2^{\frac{1}{1-h}}}+h-1$,

$$t_{hj}\overset{\eqref{eqn:12-6-1}}{=}\frac{r_jn}{h}-r_jm+\sum_{i=0}^{h-2}(h-i-1)t_{ij}$$
is a nonnegative integer for any $j\in[1,k]$. There are $\lambda\binom{n-m}{h}$ $\alpha^{h}$-edges in $\mathcal H$, and hence the following argument confirms that all $\alpha^{h}$-edges can be colored as desired.
\begin{align*}
\sum_{j=1}^{k}t_{hj}&\overset{\eqref{eqn:12-7}}{=}\sum_{j=1}^{k}\left(\frac{r_jn}{h}-\sum_{i=0}^{h-1}t_{ij}\right)=\lambda\binom{n}{h}-\lambda\sum_{i=0}^{h-1}\binom{m}{h-i}\binom{n-m}{i}\\
&\overset{\eqref{equ:comb_indetity}}{=}\lambda\binom{n}{h}-\lambda\left(\binom{n}{h}-\binom{n-m}{h}\right)=\lambda\binom{n-m}{h}.
\end{align*}
After coloring all the $\alpha^h$-edges of $\mathcal{H}$, for each $j\in[1,k]$, the total number of occurrences of $\alpha$ among all edges colored by $j$ in $\mathcal{H}$ is
\begin{align}\label{eqn:12-7-1}
\textbf{deg}_{\mathcal{H}(j)}(\alpha) & = \sum_{i=1}^hit_{ij}=
h\sum_{i=0}^{h}t_{ij}-\sum_{i=0}^{h-1}(h-i)t_{ij} \nonumber\\
& \overset{\eqref{eqn:12-6-1} \eqref{eqn:12-7}}{=}r_jn-r_jm=r_j(n-m).
\end{align}

\textbf{Step $3$.} By Lemma \ref{thm:B}, there exists an $n$-vertex hypergraph $\mathcal{F}$ obtained by replacing the vertex $\alpha$ of $\mathcal{H}$ by $n-m$ new vertices $\alpha_1,\ldots,\alpha_{n-m}$ in $\mathcal{F}$ and replacing each $X\alpha^i$-edge by an edge of the form $X\cup U$ where $U\subseteq\{\alpha_1,\ldots,\alpha_{n-m}\}$, $|U|=i\in[1,h]$ such that
\begin{itemize}
    \item[$(1)$] for $i\in[1,n-m]$ and $j\in[1,k]$,
        \begin{align*}
            \textbf{deg}_{\mathcal{F}(j)}(\alpha_i)=
            \frac{\textbf{deg}_{\mathcal{H}(j)}(\alpha)}{n-m}\overset{\eqref{eqn:12-7-1}}{=}\frac{r_j(n-m)}{n-m}=r_j;
        \end{align*}
        and for $x\in V(\mathcal G)$ and $j\in[1,k]$,
        \begin{align*}
            \textbf{deg}_{\mathcal{F}(j)}(x)=
            \frac{\textbf{deg}_{\mathcal{H}(j)}(x)}{1}\overset{\eqref{x}}{=}r_j;
        \end{align*}
  \item[$(2)$] for $U\subseteq\{\alpha_1,\ldots,\alpha_{n-m}\}$, $|U|=i\in[0,h]$, $X\subseteq V(\mathcal{G})$ and $|X|=h-i$,
        \begin{align}
            \textbf{mult}_{\mathcal{F}}(X\cup U)=\frac{\textbf{mult}_{\mathcal{H}}(X\alpha^i)}{\binom{n-m}{i}}=\frac{\lambda\binom{n-m}{i}}{\binom{n-m}{i}}=\lambda.\nonumber
        \end{align}
\end{itemize}
By $(1)$, $\mathcal{F}(j)$ is an $r_j$-factor for each $j\in[1,k]$, and by $(2)$, $\mathcal{F}\cong\lambda K_n^h$.
\end{proof}

\section{Concluding remarks}\label{sec:concluding}

For $\textbf{r}=(r_1,\ldots,r_k)$, this paper examines the conditions under which a partial $\textbf{r}$-factorization of $\lambda K_m^h$ can be extended to an $\textbf{r}$-factorization of $\lambda K_n^h$. The proof of Theorem \ref{main-r-r} mirrors that of \cite[Theorem 1.1]{arxiv202209}, yet we have refined it to be more concise and have derived an improved lower bound for $n$.

As stated in Section \ref{sec:intro}, it was pointed out in \cite[Theorem 1.7]{BN} that when $\gcd(m,n,h)$ $=\gcd(m,h)$, an $r$-factorization of $K_m^h$ can be extended to an $r$-factorization of $K_n^h$ with $n>m$ if and only if $n\geq 2m$ and obvious necessary divisibility conditions are satisfied. This implies that the tight lower bound for $n$ in Theorem \ref{main-r-r} might be $2m$. Examining the proof of Theorem \ref{main-r-r}, we observe that the bound $n>\frac{m-1}{1-2^{\frac{1}{1-h}}}+h-1$ was used to create a contradiction to \eqref{eqn:12-12}, which stems from an estimation of $\sum_{j=1}^{k}\sum_{x\in X}\textbf{deg}_{\mathcal{H}(j)}(x)$. Given our limited understanding of the structure of the given partial $\textbf{r}$-factorization of $\lambda K_m^h$, we provided a very rough lower bound \eqref{eqn:12-12-1} for $\sum_{j=1}^{k}\sum_{x\in X}\textbf{deg}_{\mathcal{H}(j)}(x)$. That is why we are unable to tighten our lower bound for $n$ to $2m$.

The essence of Step 1 in the proof of Theorem \ref{main-r-r} is coloring a non-uniform hypergraph defined on $[1,m]$ with edges from $\bigcup_{i=1}^{h-1} \lambda\binom{n-m}{h-i}\binom{[1,m]}{i}$, where $n>m$ and $\binom{[1,m]}{i}$ is the family of all $i$-subsets of $[1,m]$, each $i$-subset occurring exactly $\lambda\binom{n-m}{h-i}$ times. The interested reader is referred to \cite{hhm,JSZ} for a recent work on examining the existence of a 1-factorization of the hypergraph $GK_n^{\leq h}$ consisting of all subsets of $[1,n]$ of size up to $h$, where each $i$-subset with $i\leq h$ occurs $\lambda_i$ times as edges.




%
%
%
%
%

%
%
%
%
%
%

\end{document}